\documentclass[12pt]{article}

\usepackage{amsmath,amsthm,amsfonts,amssymb,amscd,cite}
\usepackage{graphicx,subfigure}
\usepackage[title]{appendix}
\allowdisplaybreaks[4]

\newtheorem {lemma}{Lemma}[section]
\newtheorem {theorem} {Theorem}[section]

\newtheorem {claim}{Claim}[section]
\newtheorem {proposition}{Proposition}[section]
\newtheorem {question}{Question}[section]
\usepackage[figurename=Fig.]{caption}
\usepackage{fullpage}
\usepackage{tikz}
\usetikzlibrary{positioning}
\usepackage{xcolor}
\usepackage{circuitikz}
\usetikzlibrary{patterns}
\usetikzlibrary{shapes.geometric}

\begin{document}

\title{Spectral conditions for the existence of chorded cycles in graphs with fixed size}

\author{Jin Cai\footnote{Email: jincai@m.scnu.edu.cn}, Leyou Xu\footnote{Email: leyouxu@m.scnu.edu.cn}, Bo Zhou\footnote{Email: zhoubo@m.scnu.edu.cn}\\
School of Mathematical Sciences, South China Normal University\\
Guangzhou 510631, P.R. China}

\date{}
\maketitle

\begin{abstract}
A chorded cycle is a cycle with at least one chord.
Gould  asked in [Graphs Comb. 38 (2022) 189] the question: What spectral conditions imply a graph contains a chorded cycle?
For a graph with fixed size, extremal spectral conditions are given
to ensure that a graph contains a chorded cycle and  a $(2k-3)$-chorded $(2k+1)$-cycle for $k\ge 2$, respectively, via spectral radius.
\end{abstract}

\section{Introduction}

We only consider finite simple graphs.
The problem of the existence of cycles  satisfying certain  conditions in a graph has been studied extensively, and a number of well-known results which guarantee the existence of such cycles  have been obtained \cite{BFG,BLS,Gou}.

An edge that joins two vertices
of a cycle $C$ is a chord of $C$ if the edge is not itself an edge of $C$.
A cycle with at least $k$ chords is called a $k$-chorded cycle.
A $1$-chorded cycle is known as a chorded cycle and a $2$-chorded cycle is known as a doubly chorded cycle.
To answer P\'{o}sa's
question  \cite{Po} of finding degree conditions that imply the existence of a chorded cycle,
Czipzer \cite{Cz} proved
that any graph with minimum degree at least three contains a chorded cycle  (see Lov\'{a}sz  \cite{Lov}, problem 10.2 and its solution). It may be shown that,
in fact, such a condition implies the existence of a  $2$-chorded cycle.
In this sense,
it is natural to find conditions that imply the
existence of $k$-chorded cycles with $k\ge 2$. For example, minimum degree conditions that implies the
existence vertex disjoint $2$-chorded cycles were given by Qiao and Zhang \cite{QZ}, Cream et al. \cite{CFG} and
 Santana and Van Bonn \cite{SV}, and a degree sum condition was established by Gould et al. \cite{GHH}.
There are also conditions that imply the existence of $2$-chorded cycles in  bipartite graphs  \cite{W97,GLW}.  For more results on  chorded cycles, one may refer to the recent survey \cite{Gou} (and references there), in which
Gould asked the question: What spectral conditions imply a graph contains a chorded cycle? Answers are given for graphs with fixed order via spectral radius \cite{ZHW} and signless Laplacian radius \cite{XZ}, respectively.
Motivated by the above work,
we investigate the following question:

\begin{question} \label{Qu}
What conditions via spectral radius imply a graph with fixed size contains a chorded cycle ($s$-chorded $k$-cycle with $k\ge 4$ and $1\le s<{k\choose 2}-k$, respectively)?
\end{question}

For a graph $G$, we denote by $\rho(G)$ its spectral radius, which is defined to be the largest eigenvalue of the adjacency matrix of $G$. Let $G\cup H$ be the disjoint union of graphs $G$ and $H$. The disjoint union of $k$ copies of a graph $G$ is denoted by $kG$.
The join of disjoint graphs $G$ and $H$, denoted by $G\vee H$, is the graph obtained from $G\cup H$ by adding all possible edges between vertices in $G$ and vertices in $H$. Denote by $K_n$ the $n$-vertex complete graph, and $P_n$ the $n$-vertex path.
Denote by $K_{n_1,\dots, n_k}$ the complete $k$-partite graph with classes containing $n_1,\dots, n_k$ vertices.
For simplicity, we consider isolate-free graphs, i.e., graphs without isolated vertices.



The main results are listed as below.

\begin{theorem}\label{chord}
Suppose that $G$ is an isolate-free graph with size $m$, where $m\ge 4$.

(i) If $m\le 8$ and $\rho(G)\ge \theta(m)$, then $G$ contains a chorded cycle unless $G\cong G_m$, where
 $G_m=K_1\vee tK_2$ if $m=6$,
and $G_m=K_1\vee (tK_2\cup (m-3t)K_1)$ if $m=4,5,7,8$ with $t=\lfloor\frac{m}{3}\rfloor$, and
 $\theta(m)$ is the largest root of
$x^3-x^2+(t-m)x+m-3t=0$.

(ii)  If $m\ge 9$ and $\rho(G)\ge \sqrt{m}$, then $G$ contains a chorded cycle unless
$G\cong K_1\vee 3K_2$, $K_1\vee (2K_2\cup 3K_1)$, $K_1\vee (K_2\cup 6K_1)$ or $K_{1,9}$ if $m=9$, and
$G\cong K_{1,m}$ or $K_{2,\frac{m}{2}}$ with even $m$ if $m\ge 10$.
\end{theorem}

Yu, Li and Peng \cite{YLP} showed that  for an isolate-free  graph $G$ with $m\ge 8$ edges, if $\rho(G)\ge \frac{1}{2}+\sqrt{m-\frac{3}{4}}$, then $G$ contains a copy of $K_1\vee P_4$.
Zhang and Wang \cite{ZW}) showed  same result for $m\ge 11$.
Note that there are two $2$-chorded $5$-cycles, one is $K_1\vee P_4$, the other is $SK_4$, obtained from $K_4$ by subdividing an edge. So a spectral sufficient condition for the existence of $2$-chorded cycles is known.
Recently, Li, Zhai and Shu \cite{LZS} showed the following result:
For integers $k$ and $m$ with $k\ge 3$ and $m\ge 4(k^2+3k+1)^2$, suppose that $G$ is an isolate-free graph with size $m$. If $\rho(G)\ge \frac{k-1+\sqrt{4m-k^2+1}}{2}$, then $G$ contains a  chorded $(2k+1)$-cycle or a chorded $(2k+2)$-cycle
unless $G\cong K_k\vee\left(\frac{m}{k}-\frac{k-1}{2}\right)K_1$, where the chord connects two vertices of distance two of the cycle.

\begin{theorem}\label{Cai2}
For integers  $k\ge 2$ and  $m\ge \frac{(7k^2+6k+2)^2}{16}+\frac{k^2-1}{4}$, suppose that $G$ is an isolate-free graph with size $m$. If $\rho(G)\ge \frac{k-1+\sqrt{4m-k^2+1}}{2}$, then $G$ contains a $(2k-3)$-chorded $(2k+1)$-cycle
unless $G\cong K_k\vee\left(\frac{m}{k}-\frac{k-1}{2}\right)K_1$.
\end{theorem}

Equivalently, Theorems \ref{chord} and \ref{Cai2} may be restated as the following forms.

\begin{theorem}\label{chord1}
Suppose that $G$ is  an isolate-free graph with size $m$, where $m\ge 4$. If  $G$ does not contain a chorded cycle, then
\[
\rho(G)\le \begin{cases}
\theta(m) & \mbox{if $m\le 8$}\\
\sqrt{m} & \mbox{if $m\ge 9$}
\end{cases}
\]
with equality if and only if
\begin{equation} \label{xz20}
G\cong \begin{cases}
G_m & \mbox{if $m\le 8$},\\
K_1\vee 3K_2, K_1\vee (2K_2\cup 3K_1), K_1\vee (K_2\cup 6K_1), K_{1,9} & \mbox{if $m=9$},\\
K_{1,m}, K_{2, \frac{m}{2}} & \mbox{if $m\ge 10$}.
\end{cases}
\end{equation}
\end{theorem}

\begin{theorem}\label{Cai2+}
For integers  $k\ge 2$ and  $m\ge \frac{(7k^2+6k+2)^2}{16}+\frac{k^2-1}{4}$, if  $G$ is an isolate-free graph with size $m$. If $G$ does not contain a $(2k-3)$-chorded $(2k+1)$-cycle,
then
$\rho(G)\le \frac{k-1+\sqrt{4m-k^2+1}}{2}$ with equality if and only if $G\cong K_k\vee\left(\frac{m}{k}-\frac{k-1}{2}\right)K_1$.
\end{theorem}

We mention that the results echo a series of results for graphs with fixed size in the literature, some of which are listed below.

Let $G$ be an isolate-free  graph with $m\ge 1$ edges.

\begin{enumerate}

\item[(i)] If $\rho(G)\ge \sqrt{m}$, then $G$ contains a triangle unless $G$ is a complete bipartite graph \cite{No,Nik} (see Theorems 14 and 15 in \cite{Nik}).

\item[(ii)] If $m\ge 9$ and $\rho(G)\ge \sqrt{m}$,  then $G$ contains a $4$-cycle unless $G\cong K_{1,m}, K_{1,8}^+$  \cite{Nik1}.
If $m\ge 22$ and $\rho(G)\ge \sqrt{m-1}$, then $G$ contains a $4$-cycle unless $G$ is one of the following graphs: $K_{1,m}$, $K_{1,m-1}^+$,
$K_{1,m-1}$ with a pendant edge attached to a pendant vertex, $K_{1,m-1}$ together with a $K_2$ \cite{Wan}. Here $K_{1,m-1}^+$ denote the graph obtained from $K_{1,m-1}$ by adding an edge for form a triangle.

\item[(iii)] If $m\ge 3$ and  $\rho(G) \ge
 \frac{1}{2}+\sqrt{m-\frac{3}{4}}$, then $G$ contains a $K_4$-minor unless $G\cong K_{1,1,\frac{m-1}{2}}$ \cite{Z}.

\item[(iv)]  Suppose that $\rho(G)\ge \frac{1}{2}+\sqrt{m-\frac{3}{4}}$.  If $m\ge 8$ ($m\ge 22$, respectively), then $G$ contains a $5$-cycle ($6$-cycle, respectively)
      unless $m$ is odd and $G\cong K_{1,1,\frac{m-1}{2}}$ \cite{ZLS}.

\item[(v)] Suppose that $m$ is even and $\rho(G)\ge \gamma(m)$.
If $m\ge 14$ ($m\ge 74$, respectively), then $G$ contains a $5$-cycle ($6$-cycle, respectively)
      unless $G\cong K_{1,1,\frac{m}{2}}^{-}$ \cite{MLH}.

\end{enumerate}

Note the last two items have been strengthened to contain a chorded $5$-cycle
($6$-cycle
, respectively),
see \cite{LW} and references there. For more results
in this direction, see, e.g. \cite{CDT1,CDT2,LLL}.
On the other hand, the problem of determining the graphs in
a class of (connected) graphs of fixed size that maximize the
spectral radius received much attention, see, e.g. \cite{BH,Row1,LW,LLL,MLH,ZLS}. In this sense, the class of graphs in this article is the graphs  of fixed size that do not contain a chorded cycle or a $(2k-3)$-chorded $(2k+1)$-cycle for $k\ge 2$.

\section{Preliminaries}

Let $G$ be a graph with vertex set $V(G)$ and edge set $E(G)$.
For $v\in V(G)$, we denote by $N_G(v)$ the neighborhood of $v$ in $G$, and $\delta_G(v)$ the degree of $v$ in $G$. Let $N_G[v]=\{v\}\cup N_G(v)$.  A pendant vertex is a vertex of degree one.
For a graph $G$ with $\emptyset\ne S\subseteq V(G)$, denote by $G[S]$ the subgraph of $G$ induced by $S$, and if $G[S]$ is edgeless, then $S$ is an independent set of $G$.
For simplicity, we denote $N_{G[S]}(v)$ by $N_S(v)$ and $\delta_{G[S]}(v)$ by $\delta_S(v)$ when $S\subset V(G)$. For $S,T\subset V(G)$, let $e(S,T)$ be the number of edges between $S$ and $T$. Particularly, if $S=T$, then $e(S)$ denotes the number of edges in $G[S]$.

For a graph $G$ with $E_1\subseteq E(G)$, denote by $G-E_1$ the graph with vertex set $V(G)$ and edge set
$E(G)\setminus E_1$, and in particular, we write $G-f$ for $G-\{f\}$ for $f\in E(G)$. If $E_2$ is a subset of the edge set of the complement of $G$, then $G+E_2$ denotes the graph with vertex set $V(G)$ and edge set $E(G)\cup E_2$, and in particular, we write $G+f$ for $G+\{f\}$ if $E_2=\{f\}$.
For a graph $G$ with $v\in V(G)$, $G-v$ denoted the graph obtained from $G$ by deleting $v$ and its incident edges.

For an $n$-vertex graph $G$, the adjacency matrix of $G$ is the $n\times n$ matrix
$A(G)=(a_{uv})_{u,v\in V(G)}$, where $a_{uv}=1$ if $uv\in E(G)$ and $a_{uv}=0$ otherwise.
For a positive integer $k$, let
$(a_{uv}^{(k)})_{u,v\in V(G)}$ be the  $k$-th power of $A(G)$. It is well known that
$a_{uv}^{(k)}$ counts the number of walks of length $k$ from $u$ to $v$ in $G$.

The following lemma is an immediate consequence of the Perron-Frobenius theorem.

\begin{lemma}\label{addedges}
Let $G$ be a graph and $u$ and $v$ be two nonadjacent vertices of $G$. If $G+uv$ is connected,  then $\rho(G+uv)> \rho(G)$.
\end{lemma}

If $G$ is a connected graph, then $A(G)$ is irreducible and so by Perron-Frobenius theorem, there exists a unique unit positive eigenvector of $A(G)$ corresponding to $\rho(G)$, which we call the Perron vector of $G$.

The following lemma is a special case of Lemma 6 in \cite{NR}, see also \cite{Row,SLB}.

\begin{lemma} \label{perron}
Let $G$ be a graph with $\{u,v\}\subset V(G)$ and $\emptyset\ne S\subseteq N_G(v)\setminus N_G[u]$. Let
$G'=G-\{vw: w\in S\}+\{uw: w\in S\}$. If $\mathbf{x}$ is the Perron vector of $G$ with $x_u\ge x_v$, then $\rho(G')>\rho(G)$.
\end{lemma}

For a nonnegative square matrix $M$, denote by $\lambda(M)$ its spectral radius.

Suppose that $V(G)$ is partitioned as $V_1\cup \dots\cup V_m$. For $1\le i<j\le m$, set $A_{ij}$ to be the submatrix of $A(G)$ with rows corresponding to vertices in $V_i$ and columns corresponding to vertices in $V_j$. The quotient matrix of $A(G)$ with respect to the partition $V_1\cup \dots \cup V_m$ the matrix  $B=(b_{ij})$, where $b_{ij}=\frac{1}{|V_i|}\sum_{u\in V_i}\sum_{v\in V_j}a_{uv}$. If $A_{ij}$ has constant row sum for  $1\le i<j\le m$, then we say $B$ is an equitable quotient matrix of $A(G)$.
The following lemma is an immediate consequence of \cite[Lemma 2.3.1]{BHa}

\begin{lemma}\label{quo}
For a connected graph $G$, if $B$ is an equitable quotient matrix of $A(G)$, then $\lambda(B)=\rho(G)$.
\end{lemma}

\section{Spectral conditions for existence of chorded cycles: Proof of Theorem \ref{chord1}}

\begin{proof}[Proof of Theorem \ref{chord1}]
Suppose that $G$ is an isolate-free graphs with $m$ edges  that does not contain a chorded cycle such that $\rho(G)$ is as large as possible, where $m\ge 4$.
It is easy to see (from Lemma \ref{quo} or a direct calculation) that $\rho(G_m)=\theta(m)$ for $m\le 8$,
and $\rho(K_{1,m})=\sqrt{m}$ for $m\ge 9$.  As $G_m$ for $m\le 8$ and $K_{1,m}$ for $m\ge 9$ do not have a chorded cycle, we have $\rho(G)\ge \theta(m)$ for $m\le 8$ and $\rho(G)\ge  \sqrt{m}$ for $m\ge 9$.
It suffices to show that \eqref{xz20} is satisfied.

Suppose that $G$ is disconnected. Then $\rho(G)=\rho(G_1)$ for some component $G_1$ of $G$. Let $m'=m-|E(G_1)|$. Let $v\in V(G_1)$. Let $H$ be the graph obtained from $G_1$ and $m'$ isolated vertices by adding edges to connect $v$ and all isolated vertices.
Evidently, $H$ is an isolate-free graph of size $m$ that does not contain a chorded cycle. By Lemma \ref{addedges}, $\rho(H)>\rho(G)$, which
is a contradiction. Thus, $G$ is connected.

Let $\rho=\rho(G)$ and let $\mathbf{x}$ be the Perron vector of $G$. Denote by $u$ a vertex with $x_u=\max\{x_w: w\in V(G)\}$. Let $X=N_G(u)$ and $Y=V(G)\setminus(\{u\}\cup X)$. As $G$ does not contain a  chorded cycle, $G[X]$ does not contain $P_3$ as a subgraph. So $G[X]$ consists of isolated vertices or independent edges. Denote by $X_0$ the set of isolated vertices in $G[X]$ and $X_1=X\setminus X_0$.

\begin{claim}\label{pen}
Each pendant vertex  of $G$ (if any exists) is adjacent to $u$.
\end{claim}
\begin{proof}
Suppose that $G$ has a pendant vertex $v\in Y$. Denote by $v'$ its unique neighbor. Let $G'=G-vv'+uv$.
Note that $v$ is still a pendant vertex of $G'$, so $G'$ does not contain a chorded cycle. However,  we have by Lemma \ref{perron} that $\rho(G')>\rho(G)$, a contradiction.
\end{proof}

\begin{claim}\label{a1b} If $X_1\ne \emptyset$ and $Y\ne \emptyset$, then
$e(X_1,Y)=0$.
\end{claim}
\begin{proof}
Suppose that $e(X_1,Y)\ne 0$. Let  $vw\in E(G)$ with $v\in X_1$ and $w\in Y$. Let $v'$ be the unique neighbor of $v$ in $X$.

If $w$ has a neighbor $v''$ in $X\setminus\{v\}$, then either $v''=v'$ and so $uvwv'u$ is a cycle with chord $vv'$ or
$v''\ne v'$ and so $uv'vwv''u$ is a cycle with chord $uv$, a contradiction in either case.
It follows that
$w$ does not have a neighbor in $X\setminus\{v\}$.

By Claim \ref{pen}, the component $H$ of $G[Y]$ containing $w$ can not be trivial. Suppose that $w$ is the unique vertex of $H$ that has a neighbor in $X$. Then  $vw$ is a cut edge of $G$.
Let  $G'=G-vw+uw$. Evidently, $uw$ is a cut edge of $G$, so $G'$ does not contain a chorded cycle. However,  we have $\rho(G')>\rho(G)$ by  Lemma \ref{perron}, a contradiction.
Let $w_1$ be a vertex in $V(H)\setminus\{w\}$ that   has  a neighbor, say $v_1$, in $X$.
Let $P$ be the path between $w$ and $w_1$ in $H$.
If $v_1=v'$, then $uvwPw_1v'u$ is a cycle with chord $vv'$, a contradiction. So $v_1\ne v'$.
If $v_1\ne v$, then $uv'vwPw_1v_1u$ is a cycle with chord $uv$, also a contradiction. So $v_1=v$.
From the arbitrariness  of $w_1$ it follows that $v$ is a cut vertex of $G$.
Let $G''=G-\{vz: z\in N_H(v)\}+\{uz: z\in N_H(v)\}$. Note that $u$ is a cut vertex of $G''$ and $G[\{v\}\cup V(H)]\cong G''[\{u\}\cup V(H)]$.
So $G''$ does not contain a chorded cycle. By Lemma \ref{perron}, $\rho(G'')>\rho(G)$, a contradiction. 
\end{proof}

\begin{claim}\label{bipartite}
If $X$ is independent, then $m\ge 9$ and $G\cong K_{1,m}$ or $K_{2,\frac{m}{2}}$ with even $m$.
\end{claim}
\begin{proof}
As $X$ is independent, we have $a_{uv}^{(2)}=0$ for $v\in X$.
Note also that $a_{uv}^{(2)}=e(X,Y)$ for $v\in Y$.
From $\rho^2\mathbf{x}=A(G)^2\mathbf{x}$ at $u$, we have
\begin{align*}
\rho^2x_u & =a_{uu}^{(2)}x_u+\sum_{v\in N_G(u)}a_{uv}^{(2)}x_v+\sum_{v\in V(G)\setminus (N_G(u)\cup \{u\})}a_{uv}^{(2)}x_v\\
&\le \delta_G(u)x_u+\sum_{v\in Y}a_{uv}^{(2)}x_u\\
&= (\delta_G(u)+e(X,Y))x_u.
\end{align*}
That is,
\[
\rho^2\le \delta_G(u)+e(X,Y).
\]
As $\theta(m)> \sqrt{m}$ if $m\le 8$ and  $\rho^2\ge m=\delta_G(u)+e(X,Y)+e(Y)$, we have $m\ge 9$ and
$e(Y)=0$, and $\rho^2= m$. Thus $G$ is a bipartite graph with bipartition $(X, \{u\}\cup Y)$ and $\rho=\sqrt{m}$. This implies that $G$ is a complete bipartite graph. So the result follows by noting that $K_{1,m}$ and $K_{2,\frac{m}{2}}$ with even $m$ are the unique complete bipartite graphs with no chorded cycle.
\end{proof}

\begin{claim}\label{bo}
If $|X_1|=2$ and $6\le m\le 9$, then $m=9$ and $G\cong K_1\vee (K_2\cup 6K_1)$.
\end{claim}

\begin{proof}
Suppose that $Y\ne\emptyset$. By Claim \ref{a1b}, there is a vertex $w\in Y$ adjacent to some vertex in $X_0$ such that $\delta_{X_0}(w)$ is as large as possible. As $m\le 9$ and $|X_1|=2$, we have $1\le \delta_{X_0}(w)\le 3$.

If $\delta_{X_0}(w)=3$, then $m=9$, $G\cong H_1$ (see Fig.~\ref{H123}), and it may be checked that  $\rho(G)=2.8156<3$, a contradiction.

Suppose that $\delta_{X_0}(w)=2$. Suppose that $|Y|\ge 2$. Let $z\in Y\setminus\{w\}$.
If $\delta_{X_0}(z)=2$, then $m=9$, $N_{X_0}(z)=N_{X_0}(w)$,  $G\cong H_2$ (see Fig.~\ref{H123}), so $\rho(G)=2.7321<3$, a contradiction. If $\delta_{X_0}(z)=1$, then $wz\in E(G)$ and $z$ is adjacent to one vertex in $N_{X_0}(w)$, which shows that $m=9$ and $G$ contains a chorded cycle, a contradiction. So $\delta_{X_0}(z)=0$. By  Claims \ref{pen} and \ref{a1b}, we have $m\ge 10$, a contradiction.
So $|Y|=1$, $m=7$, $G\cong H_3$ (see Fig.~\ref{H123}), and $\rho(G)=2.5035<\theta(7)$, a contradiction. %

\begin{figure}[htbp]
\centering
\begin{tikzpicture}
\draw  [black](0,0)--(-1,1)--(-1,-1)--(0,0);
\filldraw [black] (0,0) circle (2pt);
\filldraw [black] (-1,1) circle (2pt);
\filldraw [black] (-1,-1) circle (2pt);
\draw [black] (0,0)--(1,1)--(2,0)--(1,0)--(0,0);
\draw [black] (0,0)--(1,-1)--(2,0);
\filldraw [black] (1,1) circle (2pt);
\filldraw [black] (2,0) circle (2pt);
\filldraw [black] (1,0) circle (2pt);
\filldraw [black] (1,-1) circle (2pt);
\node at (0.4,-1.5) {$H_1$};
\end{tikzpicture}
\qquad
\begin{tikzpicture}
\draw  [black](0,0)--(-1,1)--(-1,-1)--(0,0);
\filldraw [black] (0,0) circle (2pt);
\filldraw [black] (-1,1) circle (2pt);
\filldraw [black] (-1,-1) circle (2pt);
\draw [black] (0,0)--(1,1)--(2,1)--(1,-1)--(0,0);
\draw [black] (1,1)--(2,-1)--(1,-1);
\filldraw [black] (1,1) circle (2pt);
\filldraw [black] (2,1) circle (2pt);
\filldraw [black] (2,-1) circle (2pt);
\filldraw [black] (1,-1) circle (2pt);
\node at (0.4,-1.5) {$H_2$};
\end{tikzpicture}
\qquad
\begin{tikzpicture}
\draw  [black](0,0)--(-1,1)--(-1,-1)--(0,0);
\filldraw [black] (0,0) circle (2pt);
\filldraw [black] (-1,1) circle (2pt);
\filldraw [black] (-1,-1) circle (2pt);
\draw [black] (0,0)--(1,1)--(2,0)--(1,-1)--(0,0);
\filldraw [black] (1,1) circle (2pt);
\filldraw [black] (2,0) circle (2pt);
\filldraw [black] (1,-1) circle (2pt);
\node at (0.4,-1.5) {$H_3$};
\end{tikzpicture}
\caption{The graphs $H_1$, $H_2$ and $H_3$.}
\label{H123}
\end{figure}
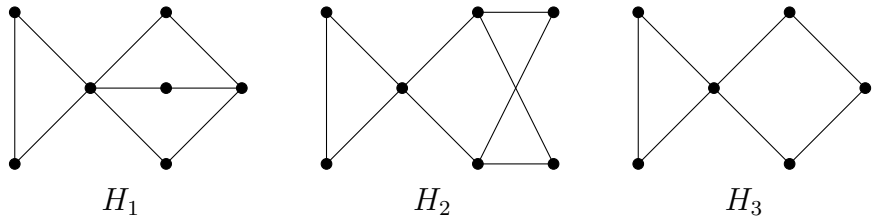

Suppose that $\delta_{X_0}(w)=1$. Denote by $v$ the neighbor of $w$ in $X_0$. If $v$ is a cut vertex of $G$, then
\[
G'=G-\{vz:z\in N_Y(v)\}+\{uz:z\in N_Y(v)\}
\]
is an isolate-free graph with $m$ edges without chorded cycle, and so by Lemma \ref{perron}, $\rho(G')>\rho(G)$, a contradiction. So $v$ lies on some cycle of $G$. It then follows that $m=8,9$ by Claims \ref{pen} and \ref{a1b}. For $m=8$, there is one possible graph $F_1$, see Fig.~\ref{F123}. For $m=9$, there are two possible graphs, $F_2$ and $F_3$, see Fig.~\ref{F123}. By direct calculation, $\rho(F_1)=2.4728<\theta(8)$, $\rho(F_2)=2.618<3$ and $\rho(F_3)=2.4562<3$, also a contradiction.

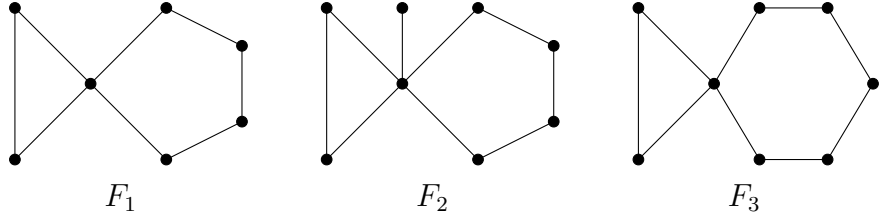
\begin{figure}[htbp]
\centering
\begin{tikzpicture}
\draw  [black](0,0)--(-1,1)--(-1,-1)--(0,0);
\filldraw [black] (0,0) circle (2pt);
\filldraw [black] (-1,1) circle (2pt);
\filldraw [black] (-1,-1) circle (2pt);
\draw [black] (0,0)--(1,1)--(2,0.5)--(2,-0.5)--(1,-1)--(0,0);
\filldraw [black] (1,1) circle (2pt);
\filldraw [black] (2,0.5) circle (2pt);
\filldraw [black] (2,-0.5) circle (2pt);
\filldraw [black] (1,-1) circle (2pt);
\node at (0.4,-1.5) {$F_1$};
\end{tikzpicture}
\qquad
\begin{tikzpicture}
\draw  [black](0,0)--(-1,1)--(-1,-1)--(0,0);
\filldraw [black] (0,0) circle (2pt);
\filldraw [black] (-1,1) circle (2pt);
\filldraw [black] (-1,-1) circle (2pt);
\draw [black] (0,0)--(1,1)--(2,0.5)--(2,-0.5)--(1,-1)--(0,0);
\filldraw [black] (1,1) circle (2pt);
\filldraw [black] (2,0.5) circle (2pt);
\filldraw [black] (2,-0.5) circle (2pt);
\filldraw [black] (1,-1) circle (2pt);
\draw [black] (0,0)--(0,1);
\filldraw[black] (0,1) circle (2pt);
\node at (0.4,-1.5) {$F_2$};
\end{tikzpicture}
\qquad
\begin{tikzpicture}
\draw  [black](0,0)--(-1,1)--(-1,-1)--(0,0);
\filldraw [black] (0,0) circle (2pt);
\filldraw [black] (-1,1) circle (2pt);
\filldraw [black] (-1,-1) circle (2pt);
\draw [black] (0,0)--(0.6,1)--(1.5,1)--(2.1,0)--(1.5,-1)--(0.6,-1)--(0,0);
\filldraw [black] (0.6,1) circle (2pt);
\filldraw [black] (1.5,1) circle (2pt);
\filldraw [black] (1.5,-1) circle (2pt);
\filldraw [black] (0.6,-1) circle (2pt);
\filldraw[black] (2.1,0) circle (2pt);
\node at (0.4,-1.5) {$F_3$};
\end{tikzpicture}
\caption{The graphs $F_1$, $F_2$, and $F_3$.}
\label{F123}
\end{figure}

Now we have proved that $Y=\emptyset$ if  $|X_1|=2$. In this case,
\[
G\cong \begin{cases}
K_1\vee (K_2\cup 3K_1)&\mbox{ if }m=6,\\
K_1\vee (K_2\cup 4K_1)&\mbox{ if }m=7,\\
K_1\vee (K_2\cup 5K_1)&\mbox{ if }m=8,\\
K_1\vee (K_2\cup 6K_1)&\mbox{ if }m=9.
\end{cases}
\]
By a direct calculation, if $6\le m\le 8$, then we have
\begin{align*}
\rho(G)&=\begin{cases}
2.5141&\mbox{ if }m=6\\
2.6813&\mbox{ if }m=7\\
2.8434&\mbox{ if }m=8
\end{cases}\\
&<\theta(m),
\end{align*}
a contradiction.  Thus  $m=9$ and $G\cong K_1\vee (K_2\cup 6K_1)$.
\end{proof}

Suppose first that $m\le 9$. By definition, $|X_1|$ is even and $|X_1|\le 6$.
If $|X_1|=6$, then $m=9$ and $G\cong K_1\vee 3K_2$.
If $|X_1|=4$, then as $G$ is connected and by Claims \ref{pen} and \ref{a1b}, we have $m=6$ and $G\cong K_1\vee 2K_2$, $m=7$ and $G\cong K_1\vee (2K_2\cup K_1)$,  $m=8$ and $G\cong K_1\vee (2K_2\cup 2K_1)$, or $m=9$ and $G\cong K_1\vee (2K_2\cup 3K_1)$.
If $|X_1|=2$, then we have by Claims \ref{pen} and \ref{a1b} that  
$G\cong K_1\vee (K_2\cup K_1)$ if $m=4$, and $G\cong K_1\vee (K_2\cup 2K_1)$ if $m=5$, and we have by Claim \ref{bo} that $m=9$,
and  $G\cong K_1\vee (K_2\cup 6K_1)$ if $m=6,\dots, 9$.
If  $X_1=\emptyset$, then we have by Claim \ref{bipartite} that $m=9$ and $G\cong K_{1,9}$. This proves \ \eqref{xz20} for $m\le 8$ and $m=9$.

Suppose next that $m\ge 10$.
It suffices to show that $X$ is independent by Claim \ref{bipartite}.
Suppose that $X$ is not independent, i.e., $X_1\ne \emptyset$.  Assume that $v_1v_2$ is an edge in $G[X]$. By Claim \ref{a1b}, $\delta_G(v_1)=\delta_G(v_2)=2$. From $\rho\mathbf{x}=A(G)\mathbf{x}$ at $v_1$ and $v_2$, respectively, we have $x_{v_1}=x_{v_2}:=x_1$ and
\begin{equation}\label{v12}
x_1=\frac{1}{\rho-1}x_u\le \frac{1}{\sqrt{m}-1}x_u.
\end{equation}
Let $G^*$ be the graph obtained from $G$ by deleting the edge $v_1v_2$ and adding a new vertex $v_3$ adjacent to $u$. Then $G^*$ is an isolate-free  graph with $m$ edges that does not contain a  chorded cycle.
Let $\mathbf{y}$ be the  vector corresponding to vertices in $G^*$ with
\[
y_w=\begin{cases}
x_w&\mbox{ if }w\in V(G)\setminus\{v_1,v_2\},\\
\frac{\sqrt{6}}{3}x_1&\mbox{ if }w=v_1,v_2,v_3.
\end{cases}
\]
Evidently, $\mathbf{y}$ is unit as $\mathbf{x}$ is unit.
Note that
\begin{align*}
\mathbf{y}^\top A(G^*)\mathbf{y}&=2\sum_{wz\in E(G^*)}y_wy_z\\
&=2\sum_{wz\in E(G)}x_wx_z-2x_1^2-4x_ux_1+2\sqrt{6}x_ux_1\\
&=\rho+2\left((\sqrt{6}-2)x_u-x_1\right)x_1\\
&\ge \rho+2\left(\sqrt{6}-2-\frac{1}{\sqrt{m}-1} \right)x_ux_1,
\end{align*}
where the last inequality follows from Eq. \eqref{v12}.
As $m\ge 10$, we have $\sqrt{6}-2-\frac{1}{\sqrt{m}-1}>0$, so by Rayleigh's principle,
\[
\rho(G^*)\ge \mathbf{y}^\top A(G^*)\mathbf{y}>\rho,
\]
a contradiction. This proves \eqref{xz20} for $m\ge 10$.
\end{proof}

\section{Spectral conditions for existence of $(2k-3)$-chorded $(2k+1)$-cycles: Proof of Theorem  \ref{Cai2+}}

We need the following lemmas.

\begin{lemma} \label{EG} \cite{BGLS} Let $G$ be a connected graph on $n$ vertices containing no path on $k+1$ vertices, where $n>k\ge 3$. Then
\[
|E(G)|\le \max\left\{\binom{k-1}{2}+n-k+1, \binom{\lceil\frac{k+1}{2}\rceil}{2}+\left\lfloor\frac{k-1}{2}\right\rfloor
\left(n-\left\lceil\frac{k+1}{2}\right\rceil\right)\right\}.
\]
If equality occurs, then $G$ is either $G_{n,k,1}$ or $G_{n,k,\lfloor\frac{k-1}{2}\rfloor}$, where
$G_{n,k,s}=(K_{k-2s}\cup (n-k+s)K_1)\vee K_s$ for $k>2s>0$.
\end{lemma}

\begin{lemma} \label{C} \cite{EG} Let $k\ge 2$ and $G$ be an $n$-vertex graph in which every cycle has length at most $k$. Then $|E(G)|\le \frac{1}{2}k(n-1)$.
\end{lemma}

\begin{lemma} \label{C+} \cite{Bo,Ore}
If $G$ is a graph on $n\ge 2$ vertices that does not have an $n$-cycle, then $|E(G)|\le {n-1\choose 2}+1$.
\end{lemma}

A $k$-core of a
graph  is a maximal induced subgraph with minimum degree is at least $k$. It is easy to see that a $k$-core of a graph $G$ can be obtained from $G$ by iteratively removing the vertices of degree at most $k-1$ until the resulting graph is trivial or is of minimum degree at least $k$. We denote by $G^c$ the $(k-1)$-core of  $G$.

\begin{proof}[Proof of Theorem  \ref{Cai2+}]
Suppose that $G$ reaches the largest spectral radius among all graphs satisfying the condition of Theorem \ref{Cai2+}.

Let $\rho=\rho(G)$ and $\mathbf{x}$ be the Perron vector of $G$. Denote by $u$ a vertex with $x_u=\max\{x_w:w\in V(G)\}$.
Let $X=N_G(u)$, 
and $Y=V(G)\setminus (\{u\}\cup X)$. Evidently, $|X|=\delta_G(u)$.

As $\rho\ge \frac{k-1+\sqrt{4m-k^2+1}}{2}$, we have $\rho^2-(k-1)\rho\ge m-\frac{k(k-1)}{2}=|X|+e(X)+e(X,Y)+e(Y)-\frac{k(k-1)}{2}$, i.e.,
\[
\rho^2-(k-1)\rho-|X|\ge e(X)+e(X,Y)+e(Y)-\frac{k(k-1)}{2}.
\]

From $\rho\mathbf{x}=A(G)\mathbf{x}$ and $\rho^2\mathbf{x}=A^2(G)\mathbf{x}$ at $u$, we have
\[
\rho x_u=\sum_{z\in X}x_z\mbox{ and }
\rho^2 x_u=\delta_G(u)x_u+\sum_{z\in X}\delta_{X}(z)x_z+\sum_{z\in Y}\delta_{X}(z)x_z,
\]
so
\[
(\rho^2-(k-1)\rho-\delta_G(u)) x_u=\sum_{z\in X}\left(\delta_{X}(z)-k+1\right)x_z+\sum_{z\in Y}\delta_{X}(z)x_z.
\]
Thus
\[
e(Y)\le \sum_{z\in X}\left(\delta_{X}(z)-k+1\right)\frac{x_z}{x_u}-e(X)+\sum_{z\in Y}\delta_{X}(z)\frac{x_z}{x_u}-e(X,Y)+\frac{k(k-1)}{2}.
\]
For each induced subgraph $H$ of $G[X]$, let
\[
\eta(H)=\sum_{z\in V(H)}\left(\delta_H(z)-k+1\right)\frac{x_z}{x_u}-|E(H)|.
\]
For convenience, we write $\eta(X)$ for $\eta(G[X])$. 
Then
\begin{equation} \label{bo}
e(Y) \le \eta(X)+\sum_{z\in Y}\delta_{X}(z)\frac{x_z}{x_u}-e(X,Y)+\frac{k(k-1)}{2}.
\end{equation}
Evidently, equality holds in \eqref{bo} if and only if $\rho^2-(k-1)\rho= m-\frac{k(k-1)}{2}$.
From \eqref{bo}, we have
\begin{equation}\label{e6}
e(Y)\le \eta(X)+\frac{k(k-1)}{2},
\end{equation}
and equality holds in \eqref{e6} if and only if $\rho^2-(k-1)\rho=m-\frac{k(k-1)}{2}$, and $x_z=x_u$ for every $z\in Y$ with $\delta_X(z)\ge 1$.
As $e(Y)\ge 0$, we have
\begin{equation}\label{e7}
\eta(X)\ge -\frac{k(k-1)}{2}.
\end{equation}

Denote by $\mathcal{C}$ the set of components of $G[X]$.
Then $\eta(X)=\sum_{H\in \mathcal{C}}\eta(H)$.
\begin{claim}\label{N2}
For each $H\in \mathcal{C}$,
$\eta(H)\le \eta(H^c)$ with equality if and only if $H=H^c$.
\end{claim}
\begin{proof}
Suppose that $H\ne H^c$. Then there is a vertex $v\in V(H)$ with $\delta_H(v)\le k-2$, so
\begin{align*}
\eta(H)-\eta(H-v)&=\sum_{z\in V(H)}(\delta_H(z)-k+1)\frac{x_z}{x_u}-|E(H)|\\
&\quad -\sum_{z\in V(H-v)}(\delta_{H-v}(z)-k+1)\frac{x_z}{x_u}+|E(H-v)|\\
&=
(\delta_H(v)-k+1)\frac{x_v}{x_u}+\sum_{z\in N_{H}(v)}\frac{x_z}{x_u}-\delta_{H}(v)\\
&\le (\delta_H(v)-k+1)\frac{x_v}{x_u}\\
&<0,
\end{align*}
which is equivalent to $\eta(H)<\eta(H-v)$. Repeating this process, we have $\eta(H)<\eta(H^c)$.
\end{proof}

\begin{claim}\label{C1}
For any $H\in \mathcal{C}$, if $H^c$ is nontrivial with $|V(H^c)|=h$, then
\[
\eta(H)\le \eta(H^c)\le
\begin{cases}
-\frac{k(k-1)}{2}& \mbox{if $h> \frac{5}{2}(k-1)$},\\
0&  \mbox{if $h=2k-1$ and $H^c$ contains a $(2k-1)$-cycle},\\
-k+1& \mbox{otherwise}.
\end{cases}
\]
Moreover, if $h>\frac{5}{2}(k-1)$ and $\eta(H^c)=-\frac{k(k-1)}{2}$, then $H^c\cong K_{k-1}\vee (n-k+1)K_1$.
\end{claim}
\begin{proof}
By Lemma \ref{N2}, we have $\eta(H)\le\eta(H^c)$. As $\frac{x_z}{x_u}\le 1$ for every $z\in V(H^c)$, we have
\[
\eta(H^c)\le 2|E(H^c)|-(k-1)h-|E(H^c)|=|E(H^c)|-(k-1)h.
\]

Suppose first that $h\ge 2k$.
As $G$ does not contain a $(2k-3)$-chorded $(2k+1)$-cycle,
$H^c$ is $P_{2k}$-free, we have
by Lemma \ref{EG} that
\begin{align*}
|E(H^c)|&\le \max\left\{\binom{2k-2}{2}+h-2k+2, \binom{k}{2}+(k-1)(h-k)\right\}\\
&=\max\left\{h+2k^2-7k+5,(k-1)h-\frac{k(k-1)}{2}\right\}\\
&=\begin{cases}
(k-1)h-\frac{k(k-1)}{2}& \mbox{ if }h> \frac{5}{2}(k-1),\\
h+2k^2-7k+5& \mbox{ if }2k\le h\le \frac{5}{2}(k-1).
\end{cases}
\end{align*}
Thus
\[
\eta(H^c)\le
\begin{cases}
-\frac{k(k-1)}{2}& \mbox{ if $h> \frac{5}{2}(k-1)$},\\
-3k+5& \mbox{ if }2k\le h\le \frac{5}{2}(k-1).
\end{cases}
\]

Suppose next that $h=2k-1$. If $H^c$ contains $(2k-1)$-cycle, then as $H$ is $P_{2k}$-free, $H=H^c$, and so $\eta(H)=\eta(H^c)\le \binom{h}{2}-(k-1)h=\binom{2k-1}{2}-(k-1)h=0$.
If $H^c$ does not contain a $(2k-1)$-cycle, then by Lemma \ref{C+},
\[
|E(H^c)|\le {2k-2\choose 2}+1=(k-1)(2k-3)+1,
\]
so $\eta(H^c)\le (k-1)(2k-3)+1-(k-1)(2k-1)=-2k+3$.

If $h\le 2k-2$, then $|E(H^c)|\le \binom{h}{2}\le(k-1)(h-1)$, so
 $\eta(H^c)\le (k-1)(h-1)-(k-1)h=-k+1$.

If $h>\frac{5}{2}(k-1)$ and $\eta(H^c)=-\frac{k(k-1)}{2}$, then by above arguments, Lemma \ref{EG} and direct calculation, $H^c\cong G_{n,2k-2,k-1}$, i.e., $H^c\cong K_{k-1}\vee (h-k+1)K_1$.
\end{proof}

If $H^c$ is trivial, then it is obvious that $\eta(H^c)< 0$. So we have by Claims \ref{N2} and \ref{C1} that $\eta(H)\le 0$ for any $H\in \mathcal{C}$.

Let $Y_1=\{y\in Y: x_y\le \frac{1}{2}x_u\}$ and $Y_2=Y\setminus Y_1$.

\begin{claim}\label{C2}
If there is a component $H\in \mathcal{C}$ such that $H^c$ contains a $(2k-1)$-cycle, then $|V(H)|=|V(H^c)|=2k-1$. Moreover,
 $\max_{v\in V(H)}x_v> \frac{1}{2}x_u$ and $e(V(H),Y_1)\le k(k-1)$.
\end{claim}
\begin{proof}
The first statement follows as $G$ does not contain a $(2k-3)$-chorded $(2k+1)$-cycle.
If $\max_{v\in V(H)}x_v\le \frac{1}{2}x_u$, then we have
\begin{align*}
\eta(X)\le \eta(H)&=\sum_{z\in V(H)}(\delta_{H}(z)-k+1)\frac{x_z}{x_u}-|E(H)|\\
&\le \frac{1}{2}\sum_{z\in V(H)}(\delta_{H}(z)-k+1)-|E(H)|\\
&=-\frac{k-1}{2}|V(H)|\\
&=-\frac{k-1}{2}(2k-1),
\end{align*}
so $\eta(X)<-\frac{k(k-1)}{2}$, which contradicts \eqref{e7}.

If $e(V(H),Y_1)> k(k-1)$, then as $x_y\le \frac{1}{2}x_u$ for each $y\in Y_1$, we have
\[
\sum_{y\in N_Y(H)}\delta_{H}(y)\left(\frac{x_y}{x_u}-1\right)\le \sum_{y\in Y_1}\delta_{H}(y)\left(\frac{x_y}{x_u}-1\right)\le -\frac{1}{2}e(V(H),Y_1)<-\frac{k(k-1)}{2}
\]
and so from \eqref{bo}, together with Claim \ref{C1}, we have
\begin{align*}
e(Y)&\le \eta(X)+\sum_{z\in Y}\delta_X(z)\frac{x_z}{x_u}-e(X,Y)+\frac{k(k-1)}{2}\\
&\le\sum_{z\in Y}\delta_X(z)\left(\frac{x_z}{x_u}-1\right)+\frac{k(k-1)}{2}\\
&\le \sum_{y\in N_Y(H)}\delta_{H}(y)\left(\frac{x_y}{x_u}-1\right)+\frac{k(k-1)}{2}\\
&<0,
\end{align*}
also a contradiction.
\end{proof}

As $\rho\ge\frac{k-1+\sqrt{4m-k^2+1}}{2}$ and $m\ge \frac{(7k^2+6k+2)^2}{16}+\frac{k^2-1}{4}$, we have
\begin{equation}\label{rho}
\rho\ge \frac{7}{4}k^2+2k.
\end{equation}

\begin{claim}\label{C5}
For any $H\in \mathcal{C}$, $H^c$ is $C_{2k-1}$-free. 
\end{claim}
\begin{proof}
Suppose that there is $H\in \mathcal{C}$ such that $H^c$ contains a $(2k-1)$-cycle.
By Claim \ref{C2}, $H=H^c$ and there is a vertex $v_0\in V(H)$ with $x_{v_0}>\frac{1}{2}x_u$. Then $\delta_{H}(v_0)\le |V(H)|-1\le 2k-2$.
By Claim \ref{C2} again, we have $\delta_{Y_1}(v_0)\le e(V(H),Y_1)\le k(k-1)$. So
\begin{align*}
\frac{1}{2}\rho x_u<\rho x_{v_0}&=x_u+\sum_{v\in N_{H}(v_0)}x_v+\sum_{z\in N_{Y_1}(v_0)}x_z+\sum_{z\in N_{Y_2}(v_0)}x_z\\
&\le \left(1+\delta_{H}(v_0)+\frac{1}{2}\delta_{Y_1}(v_0)+\delta_{Y_2}(v_0)\right)x_u\\
&\le \left(\delta_{Y_2}(v_0)+\frac{1}{2}k^2+\frac{3}{2}k-1\right)x_u.
\end{align*}
This implies that $\delta_{Y_2}(v_0)> \frac{1}{2}(\rho-k^2-3k+2)$.
Then, from  \eqref{rho}, we have $\rho\ge \frac{7}{4}k^2+2k>k^2+3k-2$, $\delta_{Y_2}(v_0)>0$.
As $G$ does not contain a $(2k-3)$-chorded $(2k+1)$-cycle, we have $N_X(y)=\{v_0\}$ for each $y\in N_{Y}(v_0)$.
As $x_y>\frac{1}{2}x_u$ for each $y\in N_{Y_2}(v_0)$,  we have from $\rho\mathbf{x}=A(G)\mathbf{x}$ at $y$ that
\[
\frac{1}{2}\rho x_u<\rho x_y=\sum_{z\in N_{H}(y)}x_z+\sum_{z\in N_{Y}(y)}x_z \le  \left(1+\delta_Y(y)\right)x_u,
\]
which implies that $\delta_Y(y)\ge\frac{1}{2}\left(\rho-2\right)$.
So
\[
\sum_{y\in N_Y(H)}\delta_Y(y)\ge \sum_{y\in N_{Y_2}(v_0)}\delta_Y(y)>\frac{1}{4}(\rho-k^2-3k+2)(\rho-2).
\]
From \eqref{rho}, we have $\rho-k^2-3k+2\ge \frac{3}{4}k^2-k+2\ge 2(k-1)$ and $\rho-2\ge \frac{7}{4}k^2+2k-2\ge 2k$. So
\[
e(Y)\ge \frac{1}{2}\sum_{y\in N_Y(H)}\delta_Y(y)>\frac{1}{8}\left(\rho-k^2-3k+2\right)\left(\rho-2\right)\ge \frac{k(k-1)}{2},
\]
contradicting the fact that $e(Y)\le \frac{k(k-1)}{2}$.
\end{proof}

For simplicity, we write $X^c$ for $G[X]^c$ and $V(G[X]^c)$.
Let $T=\{v\in X\setminus X^c: \delta_{X}(v)\le k-2\}$. 
Let $S=X\setminus (X^c\cup T)$ and $s=|S|$.

\begin{claim} \label{FK}
$\sum_{v\in S\cup X^c}\delta_X(v)\le (k-2)s+e(X^c)+e(X)$.
\end{claim}
\begin{proof}
It is trivial if $S=\emptyset$.
Suppose that $S\neq\emptyset$. Then $T\neq\emptyset$.
Note that $X^c$ can be obtained from $X$ by iteratively removing the vertices of $T\cup S$. Let $t=|T|$. Denote by $v_i$ the $i$-th vertex removed for $i=1,\dots,s+t$. Then we may assume that $T=\{v_1,\ldots,v_t\}$ and $S=\{v_{t+1},\ldots,v_{t+s}\}$.
Let $F_1=G[X]$ and $F_i=F_{i-1}-v_{i-1}$ for $i=2,\dots, s+t$. Then $\delta_{F_i}(v_i)\le k-2$ and
\[
e(S)+e(S,X^c)=\sum_{i=t+1}^{t+s}\delta_{F_i}(v_i)\le (k-2)s.
\]
So
\begin{align*}
\sum_{v\in S\cup X^c}\delta_{X}(v)&=2e(S)+e(S,T\cup X^c)+2e(X^c)+e(X^c,T\cup S)\\
&\le e(S)+e(S,X^c)+e(X^c)+e(X)\\
&\le (k-2)s+e(X^c)+e(X).\qedhere
\end{align*}
\end{proof}

By Claim \ref{FK}, we have
\begin{align*}
\eta(X)&=\sum_{v\in T}(\delta_X(v)-k+1)\frac{x_v}{x_u}+\sum_{v\in S\cup X^c}(\delta_X(v)-k+1)\frac{x_v}{x_u}-e(X)\\
&\le-\sum_{v\in T}\frac{x_v}{x_u}+\sum_{v\in S\cup X^c}\delta_{X}(v)-(k-1)(s+|X^c|)-e(X)\\
&\le -\sum_{v\in T}\frac{x_v}{x_u}+(k-2)s+e(X^c)-(k-1)(s+|X^c|)\\
&=-\sum_{v\in T}\frac{x_v}{x_u}-s+e(X^c)-(k-1)|X^c|,
\end{align*}
from which, together with \eqref{e7}, we have
\begin{equation}\label{efk}
\sum_{v\in T}\frac{x_v}{x_u}\le \frac{k(k-1)}{2}-s+e(X^c)-(k-1)|X^c|.
\end{equation}

\begin{claim}\label{C6}
For any $H\in \mathcal{C}$, if $H^c$ is non-trivial, then $\eta(H)\le \eta(H^c)\le-\frac{k(k-1)}{2}$.
\end{claim}
\begin{proof}
Suppose that there is a component $H_0\in \mathcal{C}$ such that $H_0^c$ is nontrivial and $\eta(H_0^c)>-\frac{k(k-1)}{2}$. Then by Claim \ref{C1}, $|V(H_0^c)|\le \frac{5}{2}(k-1)$ and $\eta(H^c)\le -k+1$.

From  \eqref{e7}, we have
\[
-\frac{k(k-1)}{2}\le \eta(X)=\sum_{H\in \mathcal{C}}\eta(H)\le \sum_{H\in \mathcal{C}}\eta(H^c).
\]
By Claim \ref{C1}, there is no component $H\in \mathcal{C}$ with $|V(H^c)|> \frac{5}{2}(k-1)$.
By Claim \ref{C5}, $X^c$ is $C_{2k-1}$-free and there are at most $\frac{k}{2}$ components $H\in \mathcal{C}$ with non-trivial $(k-1)$-cores $H^c$.
So $|X^c|\le \frac{k}{2}\cdot\frac{5}{2}(k-1)=\frac{5}{4}k(k-1)$.

As $G$ does not contain a $(2k-3)$-chorded $(2k+1)$-cycle, $G[X]^c$ is $C_\ell$-free for each $\ell\ge 2k$.
It then follows that every cycle in $G[X]^c$ has length at most $2k-2$ and so we have by Lemma \ref{C} that $e(X^c)\le (k-1)|X^c|$.
Then by \eqref{efk}, $\sum_{v\in T}x_v\le \left(\frac{k(k-1)}{2}-s\right)x_u$.

From $\rho\mathbf{x}=A(G)\mathbf{x}$ at $u$, we have
\[
\rho x_u=\sum_{v\in X}x_v=\sum_{v\in T}x_v+\sum_{v\in S}x_v+\sum_{v\in X^c}x_v\le \left(\frac{k(k-1)}{2}+\frac{5k(k-1)}{4}\right)x_u=\frac{7k(k-1)}{4}x_u.
\]
So $\rho\le\frac{7}{4}k(k-1)$, contradicting \eqref{rho}.
\end{proof}

\begin{claim}\label{C7}
$\eta(X)=-\frac{k(k-1)}{2}$ and
$G[X]\cong K_{k-1}\vee (\delta_G(u)-k+1)K_1$.
\end{claim}
\begin{proof}
Suppose that for each $H\in \mathcal{C}$, $H^c$ is trivial. Then  $e(X^c)=0$.
From \eqref{efk}, we have
\[
\sum_{v\in T}x_v\le\left( \frac{k(k-1)}{2}-s-(k-1)|X^c|\right)x_u\le \left( \frac{k(k-1)}{2}-s-|X^c|\right)x_u.
\]
Then from $\rho\mathbf{x}=A(G)\mathbf{x}$ at $u$,
\[
\rho x_u=\sum_{v\in T}x_v+\sum_{v\in S}x_v+\sum_{v\in X^c}x_v\le \frac{k(k-1)}{2}x_u,
\]
so $\rho\le\frac{k(k-1)}{2}$, contradicting \eqref{rho}.
Thus there is one component $H_0\in \mathcal{C}$ such that $H_0^c$ is non-trivial.
By Claim \ref{C6}, $\eta(H_0)\le -\frac{k(k-1)}{2}$.

Suppose that there is another component $H_1\in \mathcal{C}$. If $H_1^c$ is trivial, then by Claim \ref{N2}, $\eta(H_1)\le \eta(H_1^c)<0$. If $H_1^c$ is non-trivial, then by Claims \ref{N2}
and \ref{C6}, $\eta(H_1)\le \eta(H_1^c)<0$.
So $\eta(X)\le \eta(H_0)+\eta(H_1)<-\frac{k(k-1)}{2}$, contradicting \eqref{e7}.
So there is exactly one component $H_0$ in $\mathcal{C}$, i.e., $G[X]=H_0$.

From \eqref{e7}, we have $\eta(X)=\eta(H_0)=-\frac{k(k-1)}{2}$, so $X=X^c$ by Claim \ref{N2}.
If $|X|\le \frac{5}{2}(k-1)$, then, from $\rho\mathbf{x}=A(G)\mathbf{x}$ at $u$, we have
\[
\rho x_u=\sum_{v\in X}x_v\le \frac{5}{2}(k-1)x_u,
\]
so $\rho\le \frac{5}{2}(k-1)$, contradicting \eqref{rho}. This shows that $|X|>\frac{5}{2}(k-1)$. By Claim \ref{C1}, $G[X]\cong K_{k-1}\vee (\delta_G(u)-k+1)K_1$.
\end{proof}

By Claim \ref{C7}, $\eta(X)=-\frac{k(k-1)}{2}$.  So, from \eqref{e6}, we have $e(Y)=0$, which  implies that \eqref{e6} is an equality.
As $G$ is isolated-free, $\delta_X(y)\ge 1$ and $x_y=x_u$ for each $y\in Y$.

Let $V_1\subset X$ be the set of vertices of degree $\delta_G(u)-1$ in $G[X]$ and $V_2=X\setminus V_1$.
Note that
\begin{align*}
-\frac{k(k-1)}{2}=\eta(X)&=\sum_{v\in X}(\delta_{X}(v)-k+1)\frac{x_v}{x_u}-e(X)\\
&=\sum_{v\in V_1}(\delta_G(u)-k)\frac{x_v}{x_u}-\frac{k(k-1)}{2}-(\delta_G(u)-k)(k-1).
\end{align*}
Then $x_v=x_u$ for each $v\in V_1$.
From $\rho\mathbf{x}=A(G)\mathbf{x}$ at $u$ and any $v\in V_1$, we have
\[
\sum_{z\in X}x_z=\rho x_u=\rho x_v=\sum_{z\in X}x_z+\sum_{z\in N_Y(v)}x_z,
\]
so $\sum_{z\in N_Y(v)}x_z=0$. This shows that $N_Y(v)=\emptyset$ for each $v\in V_1$.

Suppose that $Y\ne\emptyset$, say $y_0\in Y$.
As $e(Y)=0$, $N_G(y_0)\subseteq V_2$.
From $\rho\mathbf{x}=A(G)\mathbf{x}$ at $y_0$, we have
\[
\rho x_{y_0}=\sum_{v\in N_G(y_0)}x_v<\sum_{v\in X}x_v=\rho x_u,
\]
showing that $x_{y_0}<x_u$, which contradicts the fact that $x_y=x_u$ for each $y\in Y$.
So $Y=\emptyset$. Thus $G\cong K_k\vee (\delta_G(u)-k+1)K_1$.

As $G$ has $m$ edges, $\delta_G(u)=\frac{m}{k}+\frac{k-1}{2}$. Therefore, $G\cong K_k\vee\left(\frac{m}{k}-\frac{k-1}{2}\right)K_1$.

Let $U_1$ be the set of vertices in $V(K_k)$ and $U_2=V(G)\setminus U_1$. Corresponding to the partition $V(G)=U_1\cup U_2$, $A(G)$ has an equitable quotient matrix
\[
B=\begin{pmatrix}
k-1 &  \frac{m}{2}-\frac{k-1}{2}\\
k & 0
\end{pmatrix}.
\]
By Lemma \ref{quo} and direct calculation, we have $\rho(G)=\frac{k-1+\sqrt{4m-k^2+1}}{2}$. This proves Theorem  \ref{Cai2+}.
\end{proof}

\section{Concluding remarks}

There are various conditions concerning the existence of chorded cycles. These conditions include
minimum degree condition, size condition,  spectral condition, etc.


Given a graph $G$ of order $n$ and size $m$, Posa \cite{Po} showed  that if $m\ge 2n-3\ge 5$ then $G$ contains a chorded cycle.

\begin{proposition} \label{RE}
Let $G$ be a graph of order $n$ and size $m$ that does not contain a $2$-chorded cycle. Then $m\le 2n-3$ and the bound is sharp.
\end{proposition}

\begin{proof}
We prove by induction on $n$. If $n=3$, then $m\le 3$, as desired. Suppose that $n\ge 4$ and the result is true for any graph of order less than $n$. As $G$ does not contain a $2$-chorded cycle, $\delta(G)\le 2$. Let $v$ be the vertex in $G$ of minimum degree. By induction, $|E(G-v)|\le 2(n-1)-3$. So $m=|E(G-v)|+\delta_G(v)\le 2(n-1)-3+2=2n-3$, as desired.
Note that $K_{1,1,n-2}$ is an $n$-vertex graph containing no $2$-chorded cycle which has $2n-3$ edges. So the bound is sharp.
\end{proof}

Let $G$ be a graph of order $n$ and size $m$. By Proposition \ref{RE},
if  $m=2n-3$, then $G$ does not necessarily contain a $2$-chorded cycle, while
if $m\ge 2n-2\ge 6$, then $G$
contains a $2$-chorded cycle.

The graphs considered in this paper have fixed size but the order is not fixed.
It would be interesting to investigate the connections among these conditions that imply the existence of a ($2$-)chorded cycle.

\bigskip
\bigskip

\noindent {\bf Acknowledgement.}
This work was supported by the National Natural Science Foundation of China (No.~12071158).

\end{document}